\documentclass[reqno, 12pt]{amsart}
\usepackage{verbatim}
\usepackage{amssymb}
\usepackage{enumerate}
\usepackage[active]{srcltx}
\numberwithin{equation}{section}

\usepackage{t1enc}
\usepackage[utf8x]{inputenc}

\newtheorem{theorem}{Theorem}[section]
\newtheorem{lemma}[theorem]{Lemma}
\newtheorem{corollary}[theorem]{Corollary}

\theoremstyle{definition}

\newtheorem{definition}[theorem]{Definition}

\newcommand{\mc}[1]{\mathcal{#1}}

\newcommand{\setm}{\setminus}
\newcommand{\empt}{\emptyset}
\newcommand{\subs}{\subset}

\def\<{\left\langle}
\def\>{\right\rangle}
\def\br#1;#2;{\bigl[ {#1} \bigr]^ {#2} }

\newcommand{\dis}{\mc D}
\newcommand{\diso}{\dis_{\omega}}

\newcommand{\rcp}[1]{\operatorname{RC}^+(#1)}

\author[I. Juh\'asz]{Istv\'an Juh\'asz}
\address      { Alfréd Rényi Institute of Mathematics, Hungarian Academy of Sciences }
\email{juhasz@renyi.hu}

\author[L. Soukup]{Lajos Soukup}
\thanks
  {
   }
\address
      { Alfr{\'e}d R{\'e}nyi Institute of Mathematics, Hungarian Academy of Sciences }
\email{soukup@renyi.hu}

\author[Z. Szentmikl\'ossy]{Zolt\'an Szentmikl\'ossy}
\address{E\"otv\"os University of Budapest}
\email{szentmiklossyz@gmail.com
}


\title[Resolvability of L- and e-generated spaces]{On the resolvability of Lindelöf-generated and (countable extent)-generated spaces}
\date{\today}
\thanks{The research on and preparation of this paper was
supported by NKFIH grant no. K113047.}

\begin{document}

\begin{abstract}
Given a topological property $P$, we say that the space $X$ is $P$-{\em generated} if for
any subset $A\subs  X$ that is not open in $X$ there is a subspace $Y \subs X$ with property $P$
such that $A\cap Y$ is not open in $Y$. (Of course, in this definition we could replace "open" with "closed".)

\smallskip

In this paper we prove the following two results:
\begin{enumerate}[(1)]
\item Every Lindelöf-generated regular space  $X$ satisfying $|X|=\Delta(X)={\omega}_1$ is ${\omega}_1$-resolvable.
\item  Any (countable extent)-generated regular space  $X$ satisfying $\Delta(X)>{\omega}$ is ${\omega}$-resolvable.
\end{enumerate}

\smallskip

These are significant strengthenings of our earlier results from \cite{JSSz2} which can be obtained from (1) and (2)
by simply omitting the "-generated" part. Moreover, the second result improves a recent result of Filatova and Osipov
from \cite{FO} which states that Lindelöf-generated regular spaces of uncountable dispersion character are 2-resolvable.
\end{abstract}

\maketitle

\section{Introduction and preliminaries}

Since this work about resolvability is to appear in a volume dedicated to the memory of Wis Comfort, it is very appropriate to start by
acknowledging that our interest in this topic had been initiated by the great survey paper \cite{CG} by Wis and his student
Garcia-Ferreira.

Now, let $\mathcal{C}$ be a topological property, or equivalently: a class of topological spaces. For any space $X$ we shall then denote by
$\mathcal{C}(X)$ the family of all subspaces of $X$ having property $\mathcal{C}$, i.e. $$\mathcal{C}(X) = \{Y \subs X : Y \in \mathcal{C} \}.$$

\begin{definition}\label{def:P}
The space $X$ is said to be $\mathcal{C}$-{\em generated} provided that, for
any subset $A\subs  X$, if $A\cap Y$ is open (resp. closed) in $Y$ for all $Y \in \mathcal{C}(X)$ then $A$ is actually open (resp. closed) in $X$.
\end{definition}

Intuitively this means that the family of subspaces $\mathcal{C}(X)$ determines the topology of $X$.
We shall use $\widetilde{\mathcal{C}}$ to denote the class of  $\mathcal{C}$-generated spaces.
Trivially, every space in $\mathcal{C}$ is also
$\mathcal{C}$-generated, i.e. $\mathcal{C} \subs \widetilde{\mathcal{C}}$. If $\mathcal{C}$ is closed hereditary then so is $\widetilde{\mathcal{C}}$,
moreover if $\mathcal{C}$ is closed hereditary then every open subspace of a {\em regular} $\mathcal{C}$-generated space is again $\mathcal{C}$-generated.

For example, compact-generated spaces coincide with what is known as $k$-spaces. We mention here that Pytkeev proved in \cite{P}
that Hausdorff $k$-spaces are maximally resolvable, i.e. $\Delta(X)$-resolvable, where $\Delta(X)$, the dispersion character of $X$,
denotes the smallest cardinality of a non-empty open set in $X$. The fact that compact Hausdorff spaces are maximally resolvable
had essentially been already established by Hewitt in his work \cite{H} that started the study of resolvability.

Motivated by the fact, also due to Hewitt, that countable regular irresolvable spaces exist,
Malychin asked in \cite{M} if regular Lindelöf spaces of uncountable dispersion character are 2-resolvable.
This question was answered affirmatively by Filatova in \cite{F} and this result was strengthened in \cite[Theorem 3.1]{JSSz2} to:

\begin{theorem}\label{tm:e}
Every regular space of countable extent and uncountable dispersion character is $\omega$-resolvable.
\end{theorem}

\noindent(A space has countable extent if every uncountable subset has an accumulation point in it.)

It was also proved in \cite[Theorem  4.3]{JSSz2} that

\begin{theorem}\label{tm:L}
Every Lindelöf regular space $X$ satisfying $|X|=\Delta(X)={\omega}_1$ is even ${\omega}_1$-resolvable.
\end{theorem}

Recently, in \cite{FO}, Filatova and Osipov have extended Filatova's above result from Lindelöf to Lindelöf-generated spaces.
The aim of this paper is to show that our above results also admit such an extension.

\medskip


Let us now fix some notation. For a topological space $X$,
we denote by $\tau^+(X)$ the collection of all non-empty open sets in $X$ and by
$\rcp{X}$ the family of all non-empty regular closed subsets of $X$.
A $\pi$-network $\mathcal{N}$ of $X$ is a family of non-empty sets such that every
$U \in \tau^+(X)$ includes a member of $\mathcal{N}$. Thus, if $X$ is regular then $\rcp{X}$ is a $\pi$-network in $X$.

For $A\subs  X$ we denote by $A'$ the derived set of $A$ consisting of all accumulation points of $A$ and by $ A^\circ$
the set of all complete accumulation points of $A$ in $X$.

If $X$ is a space and ${\kappa}$ is a cardinal then we write
\begin{displaymath}
 \dis_{\kappa}(X)=\{D\in \br X;{\kappa}; : \text{$D$ is discrete}\},
\end{displaymath}

moreover,

\begin{displaymath}
\mathfrak{R}_\kappa(X) = \{Z\subs X : Z\,\, \text{is ${\kappa}$-resolvable}\}.
\end{displaymath}

Our notation concerning cardinal functions is standard, as e.g. in \cite{J}.

To conclude this section, we list a few results about resolvability which will be
used in our proofs later. All but the last two are old and well-known.

\begin{theorem}[El'kin, \cite{El}]\label{tm:el}
If $X$ is a topological space, ${\kappa}$ is a cardinal, moreover $\mathfrak{R}_\kappa(X)$
is a $\pi$-network in $X$, then $X$ is ${\kappa}$-resolvable.
\end{theorem}

\begin{theorem}[Pytkeev, \cite{P}]\label{tm:Pt}
Every Hausdorff space $X$ with $t(X) = \omega < \Delta(X)$ is maximally, and hence $\omega_1$-resolvable.
\end{theorem}

We say that $D \subs X$ is {\em strongly discrete} if the points in $D$ can be separated by
pairwise disjoint open sets. Note that in a regular space every countable discrete set is strongly discrete.
A $T_1$ space $X$ is called an SD-space if every point
in $X$ is the accumulation point of a strongly discrete set $D \subs X$.

\begin{theorem}[\cite{Sh} or \cite{JSSz1}]\label{tm:sh}
Every SD-space is $\omega$-resolvable. In particular, if $X$ is regular and every point of $X$ is
the accumulation point of a countable discrete set then $X$ is $\omega$-resolvable.
\end{theorem}

A set $H \subs X$ is called $\kappa$-approachable in $X$ if there are $\kappa$ many {\em pairwise disjoint}
sets $\{X_\alpha : \alpha < \kappa\} \subs [X]^\kappa$ such that
for all $\alpha < \kappa$ we have $$\forall\,Y \in [X_\alpha]^\kappa\,(Y^\circ \ne \emptyset) \text{ and } {X_\alpha}^\circ = H\,.$$

The following statement is the $\kappa = \omega_1$ particular case of Lemma 2.6.(2) in \cite{JSSz2}.
Actually, in \cite{JSSz2} it is assumed that $X$ has countable extent
but it is easy to check that this assumption is not used in the proof of part (2) of Lemma 2.6 of \cite{JSSz2}.

\begin{theorem}\label{tm:app}
For any space $X$ and $A \in [X]^{\omega_1}$, if $|A^\circ| \le \omega$  and $B^\circ \ne \emptyset$
for all $B \in[A]^{\omega_1}$ then $A^\circ$ has a subset which is $\omega_1$-approachable in $X$.
\end{theorem}

The following result, that will play an essential role below, is a modification of the $\kappa = \omega_1$ particular case
of Theorem 2.7 in \cite{JSSz2}.

\begin{theorem}\label{tm:main}
Let $\mathcal{M}$ with $|\mathcal{M}| \le \omega_1$ be a family of subspaces of a regular space $X$ such that for every
$M \in \mathcal{M}$ there is a fixed $H_M \subs M$ that is $\omega_1$-approachable in $M$.
Assume, moreover, that we may simultaneously fix disjoint families $\{X_{M,\alpha} : \alpha < \omega_1\} \subs [M]^{\omega_1}$
witnessing the $\omega_1$-approachability of $H_M$ in $M$ in such a way that if $\{M,\,N \} \in [\mathcal{M}]^2$ and
$\alpha, \beta < \omega_1$ then $|X_{M,\alpha} \cap X_{N,\beta}| \le \omega$.

Then there are pairwise disjoint sets $D_\alpha \subs \bigcup \mathcal{M}$ for $\alpha < \omega_1$ such that
$\bigcup \{H_M : M \in \bigcup \mathcal{M}\} \subs \overline{D_\alpha}$ for all $\alpha < \omega_1$.
\end{theorem}

\begin{proof}
Observe that, by assumption, the collection $$\mathcal{X} = \{X_{M,\alpha}  : M \in \mathcal{M}, \alpha < \omega_1\}$$ is almost disjoint.
But $|\mathcal{X}| = \omega_1$ then implies that $\mathcal{X}$ is even essentially
disjoint, i.e. every $X_{M,\alpha} \in \mathcal{X}$ has a {\em co-countable} subset $Y_{M,\alpha}$ such that the family
$\mathcal{Y} = \{Y_{M,\alpha}  : M \in \mathcal{M}, \alpha < \omega_1\}$ is already disjoint.
Then for every pair $\langle M,\alpha \rangle \in \mathcal{M} \times \omega_1$ we clearly have $M \cap (X_{H,\alpha})^\circ = M \cap (Y_{H,\alpha})^\circ = H_M$.
But then if we put $D_\alpha = \bigcup_{M \in \mathcal{M}}Y_{M,\alpha}$ for any $\alpha < \omega_1$, then we have $\bigcup \{H_M : M \in \mathcal{M}\} \subs \overline{D_\alpha}$,
hence the pairwise disjoint sets $D_\alpha$ for $\alpha < \omega_1$ are as required.

\end{proof}

Of course, if $\bigcup \{H_M : M \in \mathcal{M}\}$ is dense in $X$ then this
yields the $\omega_1$-resolvability of $X$.
The final result of this introductory section will be used to obtain such situations.

\begin{theorem}\label{tm:pinw}
Assume that $\mathcal{L}$ is a $\pi$-network of a regular space $X$ such that every $L \in \mathcal{L}$
has a subset $H_L$ that is $\omega_1$-approachable in $L$,  witnessed by
the disjoint family $\{X_{L,\alpha} : \alpha < \omega_1\} \subs [L]^{\omega_1}$.
Let $\mathcal{K}$ be a {\em maximal} subfamily of $\mathcal{L}$
such that for any two $K, L\ \in \mathcal{K}$ and any $\alpha, \beta < \omega_1$ we have $|X_{K,\alpha} \cap X_{L,\beta}| \le \omega$.
Then $\bigcup \{H_K : K \in \mathcal{K}\}$ is dense in $X$.
\end{theorem}

\begin{proof}
For any $U \in \tau^+(X)$ there is $L \in \mathcal{L}$ such that $\overline{L} \subs \mathcal{L}$.
By the maximality of $\mathcal{K}$ then there are $K \in \mathcal{K}$ and $\alpha, \beta < \omega_1$ for which
$|X_{K,\alpha} \cap X_{L,\beta}| = \omega_1$. Put $C = X_{K,\alpha} \cap X_{L,\beta}$.
Then, on one hand, we have $C^\circ \cap H_K \ne \emptyset$, and on the other $C^\circ \subs \overline{L} \subs U$.
Consequently, we have $\emptyset \ne C^\circ \cap H_K \subs U \cap H_K$.

\end{proof}

\section{Lindelöf-generated spaces}

\bigskip

For the sake of brevity, in the rest of our paper {\em space} always means {\em regular space}.
The aim of this section is to prove the following strengthening of Theorem \ref{tm:L}.

\begin{theorem}\label{tm:lindelof-generated}
Any Lindelöf-generated space  $X$ with $|X|=\Delta(X)={\omega}_1$ is ${\omega}_1$-resolvable.
\end{theorem}

\begin{proof}
Clearly, every $Y \in \tau^+(X)$ satisfies the assumptions of our theorem, hence by Theorem
\ref{tm:el} it suffices to show that $\mathfrak{R}_{\omega_1}(X) \ne \emptyset$.
Consequently, we may assume that $T = \{x \in X : t(x,X) = \omega_1\}$ is dense in $X$,
since otherwise there is some $U \in \tau^+(X)$ with $t(U) = \omega < \Delta(U) = \omega_1$,
and hence $U \in \mathfrak{R}_{\omega_1}(X)$ by Theorem \ref{tm:Pt}.

Let us now put $\mathcal{L} = \{L \in [X]^{\omega_1} : L \text{ is Lindelöf}\}$.
If there is $L \in \mathcal{L}$ such that $\Delta(L) = \omega_1$ then by Theorem \ref{tm:L}
we have $L \in \mathfrak{R}_{\omega_1}(X)$ and we are done.
Consequently, we may assume that for every $L \in \mathcal{L}$ we have $\Delta(L) < \omega_1$.
We claim that then every $L \in \mathcal{L}$ has a subset $H$ which is $\omega_1$-approachable in $L$.

Indeed, for every $L \in \mathcal{L}$ there is some $U \in \tau^+(X)$ such that
$0 < |U \cap L \cap L^\circ| \le \omega$. If $|L \cap L^\circ| \le \omega$ then
we may choose $U = X$. Of course, $L \cap L^\circ \ne \emptyset$ because $L$ is Lindelöf.
And if $|L \cap L^\circ| = \omega_1$ then $L \cap L^\circ \in \mathcal{L}$,
hence $\Delta(L \cap L^\circ) < \omega_1$ yields such an open set $U$.

Then we may pick a point $x \in U \cap L \cap L^\circ$  and
an open set $V$ such that $x \in V \subs \overline{V} \subs U$. But then $A = L \cap \overline{V}$
is uncountable and Lindelöf, moreover $L \cap A^\circ \subs U \cap L \cap L^\circ$ is countable.
This means that $A$ has only countably many complete accumulation points in $L$.
But then $A$ satisfies the assumptions of Theorem \ref{tm:app} in $L$, and thus
$L \cap A^\circ$ has a subset $H$ which is $\omega_1$-approachable in $L$.

Next we show that $\mathcal{L}$ is a $\pi$-network of $X$.
This is where we shall use that $T$ is dense in $X$.
So, pick any $U \in \tau^+(X)$ and a point $x \in T \cap U$. We may then pick an open set $V$ such that $x \in V \subs \overline{V} \subs U$.
Then $t(x,X) = t(x,\overline{V}) = \omega_1$ clearly implies the existence of a set $A \subs \overline{V}$ that is
$\omega$-closed, i.e. for every $B \in [A]^\omega$ we have $\overline{B} \subs A$, but $x \in \overline{A} \setm A$.
Thus $A$ is not closed, hence there is a Lindelöf set $L \subs X$ such that $L \cap A$ is not closed in $L$.
Since $A \subs \overline{V}$ we may clearly assume that $L \subs \overline{V} \subs U$ as well.
But then $L \cap A$ cannot be countable because $A$ is $\omega$-closed, hence $L  \in \mathcal{L}$.

Let us now fix for each $L \in \mathcal{L}$ a subset $H_L$ which is $\omega_1$-approachable in $L$, as well as
a disjoint family $\{X_{L,\alpha} : \alpha < \omega_1\} \subs [L]^{\omega_1}$
witnessing the $\omega_1$-approachability of $H_L$ in $L$. Next, let $\mathcal{K}$ be a {\em maximal} subfamily of $\mathcal{L}$
such that if $K,\,L$ are distinct members of $\mathcal{K}$ then $|X_{K,\alpha} \cap X_{L,\beta}| \le \omega$ for all $\alpha, \beta < \omega_1$.
By Theorem \ref{tm:pinw} then $D = \bigcup \{H_K : K \in \mathcal{K}\}$ is dense in $X$.

But $|D| \le |X| = \omega_1$ clearly implies that there is $\mathcal{M} \subs \mathcal{K}$ such that $D = \bigcup \{H_M : M \in \mathcal{M}\}$.
Now, we may apply Theorem \ref{tm:main} to the family $\mathcal{M}$ to obtain
the pairwise disjoint sets $D_\alpha$ for $\alpha < \omega_1$ satisfying $D \subs \overline{D_\alpha}$.
This, however, means that each $D_\alpha$ is dense in $X$, consequently $X$ is ${\omega}_1$-resolvable.

\end{proof}

\section{(Countable extent)-generated spaces}

\bigskip

The aim of this section is to prove the following result.

\begin{theorem}\label{tm:c.e.-generated}
Any (countable extent)-generated space  $X$ satisfying $\Delta(X)>{\omega}$ is ${\omega}$-resolvable.
\end{theorem}

However, before starting the proof of this we shall need to present a number of preparatory lemmas.
We start with a definition.

\begin{definition}\label{def:good}
We say that an indexed partition $\{Z_i:i<n\}$ of the space $Z$ into finitely many subsets
is {\em good} iff
\begin{align*}
 \forall i< n\ \forall D\in \diso(Z_i)\ D'\cap \bigcup_{j\le i}Z_j=\empt.
\end{align*}
\end{definition}

Clearly, if $\{Z_i:i<n\}$ is a good partition of $Z$ and $Y \subs Z$
then $\{Y \cap Z_i:i<n\}$ is a good partition of $Y$.

\begin{lemma}\label{lm:goodpartition}
Let $X$ be any space with $\mathfrak{R}_\omega(X) = \emptyset$. Then $X$ has a non-empty open subspace $Z$
which admits a good partition $Z = \cup \{Z_i:i<n\}$. Consequently, the open subsets of $X$ admitting a good partition
form a $\pi$-base in $X$.
\end{lemma}

\begin{proof}
Let us define for any space $Y$ the operation $S(Y)$ by the formula $$S(Y) = Y \setm \cup \{D' : D \in \mathcal{D}_\omega(Y)\}\,.$$
In other words, $S(Y)$ is the set of all points in $Y$ to which no countable discrete set accumulates in $Y$.
Note that if $S(Y)$ is not dense in $Y$ then in the non-empty open subspace $U = Y \setm \overline{S(Y)}$ every point is the limit
of a countable discrete set, hence by Theorem \ref{tm:sh} $U$ is $\omega$-resolvable.

Since now we have $\mathfrak{R}_\omega(X) = \emptyset$, it follows that $S(Y)$ is dense in $Y$ for every $Y \subs X$.
Let us define the sets $\{Y_i : i < \omega\}$ with the following recursive formula:
$$ Y_i=S(X\setm \bigcup_{j<i}Y_j).$$
Then $Y_0 = S(X)$ is dense in $X$ and the sets $Y_i$ are pairwise disjoint, hence they cannot all be dense in $X$. Thus  there is a smallest
index $n > 0$ such that $X\setm \bigcup_{j<n}Y_j$ is not dense in $X$.  Clearly, then
$Z = Int(\bigcup_{j<n}Y_j)$ is as required with $Z_i = Z \cap Y_i$ for $i < n$.

This completes the proof of the first part of our lemma.
But we may clearly apply the first part of the lemma to any non-empty open subspace of $X$ to obtain the second.

\end{proof}

The following lemma actually does not need the regularity of $X$.

\begin{lemma}\label{lm:s<z}
If $X$ is any space satisfying $s(X) = \omega < d(X)$ then $\mathfrak{R}_{\omega_1}(X) \ne \emptyset$.
\end{lemma}

\begin{proof}
Let $\mathcal{G}$ be a maximal disjoint collection of countable open sets in $X$. Then $\mathcal{G}$ is
countable because $s(X) = \omega$ implies that $X$ is CCC, hence $\cup \mathcal{G}$ is
countable as well. But then $\cup \mathcal{G}$ is not dense in $X$ and so $U = Int(X \setm \overline{\cup \mathcal{G}}) \ne \emptyset$.
By the maximality of $\mathcal{G}$ we have $\Delta(U) > \omega$ and this, together with $s(U) = \omega$ implies that $U$ is maximally resolvable, in view of Theorem 2.2
of \cite{JSSz}. But then $\Delta(U) \ge \omega_1$ implies that $U \in \mathfrak{R}_{\omega_1}(X)$.

\end{proof}

Let us denote by $\mathcal{E}$ the class of all spaces that have countable extent. Then, by our introductory convention,
$\widetilde{\mathcal{E}}$ denotes the class of all (countable extent)-generated spaces. Below we define certain subclasses of
$\mathcal{E}$ that will play an important role in the proof of Theorem \ref{tm:c.e.-generated}.

\begin{definition}\label{def:Z1}

\smallskip

\begin{enumerate}[(1)]
\item We shall denote by $\mathcal{Z}_1$ the class of all spaces $Z \in \mathcal{E}$ that
have a {\em dense proper} subset $A$ that is $\omega$-closed in $Z$.

\smallskip

\item $\mathcal{Z}_2$ will denote the class of
those $Z \in \mathcal{Z}_1$ for which $I(Z)$, the set of isolated points
of $Z$, is an $\omega$-closed dense subset of $Z$, moreover $|I(Z)| = \omega_1$.

\smallskip

\item  $\mathcal{Z}_3 = \{Z \in \mathcal{Z}_2 : |I(Z)^\circ| \le \omega\}$.
\end{enumerate}

\end{definition}

The following simple lemma will turn out to be useful.

\begin{lemma}\label{lm:eZ1}
For any space X, assume that $L \in \mathcal{E}(X)$ and $A$ is an $\omega$-closed subset of $X$
such that $A \cap L$ is not closed in the subspace $L$. Then $\overline{A \cap L} \cap L \in \mathcal{Z}_1$.
\end{lemma}

\begin{proof}
Clearly, $\overline{A \cap L} \cap L \in \mathcal{E}$ because it is closed in $L$,
moreover, $A \cap L$  is a proper dense subset of $\overline{A \cap L} \cap L$. Finally, $A \cap L$ is $\omega$-closed in
$\overline{A \cap L} \cap L$. Indeed, this is because if $B \in [A \cap L]^\omega$ then
$\overline{B} \cap L$, the closure of $B$ in $L$, is included in  $A \cap L$.

\end{proof}

Note that if $Z \in \mathcal{Z}_2$ then the $\omega$-closedness of $I(Z)$ in $Z$ means that every countable subset
of $I(Z)$ is closed discrete in $Z$ and hence for every uncountable $E \subs I(Z)$ we have $E' = E^\circ \ne \emptyset$
because $Z$ has countable extent. Consequently, Theorem \ref{tm:app} implies that if $Z \in \mathcal{Z}_3$ then
$Z^\circ$ has a subset $H_Z$ which is $\omega_1$-approachable in $Z$. Moreover, as $Z^\circ$ is countable,
the disjoint family witnessing this can clearly be chosen in $[I(Z)]^{\omega_1}$.

\begin{lemma}\label{lm:Z1Z3}
Assume that $Z \in \mathcal{Z}_1$ with $A \subs Z$ being a proper dense $\omega$-closed set in $Z$.
If, in addition, $Z$ admits a good partition $\{Z_i : 0 \le i \le k\}$ and $\mathfrak{R}_{\omega}(Z) = \emptyset$,
then there is some $E \in \mathcal{D}_{\omega_1}(A)$ such that $\overline{E} \in \mathcal{Z}_3$.
\end{lemma}

\begin{proof}

{\bf Step 1.} There is some $D \in \mathcal{D}_{\omega_1}(A)$ such that $\overline{D} \in \mathcal{Z}_2$.

\smallskip

Note that every $E \in \mathcal{D}_\omega(Z_k)$ is closed in $Z$ by the definition of good partitions.
This implies that if $s(A \cap Z_k) > \omega$ then for any $D \in \mathcal{D}_{\omega_1}(Z_k)$
we have $\overline{D} \in \mathcal{Z}_2$. Thus we may assume that $s(A \cap Z_k) = \omega$.
On the other hand, we clearly have $d(A) > \omega$,
consequently lemma \ref{lm:s<z} implies that $s(A) > \omega$ as well.

Let us denote by $\mathcal{S}$ the family all those regular closed subsets $R$ of $Z$ for which
$A \cap R \ne R$. The above argument then yields $s(A \cap R) > \omega$ for all $R \in \mathcal{S}$.
This, in turn, implies that for every $R \in \mathcal{S}$
there is some $j < k$ such that $s(A \cap R \cap Z_j) > \omega$.
Let us then put
\begin{align*}
i(R) =\max \{j : s(A\cap R \cap Z_j)>{\omega}\}.
\end{align*}

Then we have $i(R) < k$ and we may choose $Q \in \mathcal{S}$ so that the value of $i(Q)$ is minimal among all
members of $\mathcal{S}$.
For each $j$ with $i(Q) < j \le k$ we have $s(A \cap Q \cap Z_j) = d(A \cap Q \cap Z_j) = \omega$ by lemma \ref{lm:s<z}. Consequently,
we may pick a countable dense subset $H$ of $A \cap Q \cap \bigcup_{i(Q) < j \le k}Z_j$.
We may pick a point $p \in Q \setm A$ and then $p \notin \overline{H} \subs A$. Thus we may find an open neighborhood
$U$ of $p$ with $\overline{U} \cap \overline{H} = \emptyset$. But then $p \in \overline{Q \cap U} \setm A$ implies
$\overline{Q \cap U} \in \mathcal{S}$, hence the trivial inequality $i(Q) \ge i(\overline{Q \cap U})$ implies $i(Q) = i(\overline{Q \cap U})$.

Then $s(A \cap \overline{Q \cap U} \cap Z_{i(Q)}) > \omega$ and for any $$D \in \mathcal{D}_{\omega_1}(A \cap \overline{Q \cap U} \cap Z_{i(Q)})$$
we have $\overline{D} \in \mathcal{Z}_2$. Indeed, this is because for any $E \in [D]^\omega$ we have
both  $E' \subs \overline{U}$ and $E' \subs A \cap Q \cap \bigcup_{i(Z) < j \le k}Z_j \subs \overline{H}$,
while $\overline{U} \cap \overline{H} = \emptyset$. Hence Step 1. is proven.

\medskip

{\bf Step 2.} Assume that $D \in \mathcal{D}_{\omega_1}(Z)$ is such that $\overline{D} \in \mathcal{Z}_2$.
Then $\overline{D} = D \cup D^\circ$ and $D^\circ \in \mathcal{E}$ being closed in $Z \in \mathcal{E}$.
But $D^\circ \subs Z$ is not $\omega$-resolvable, hence $\Delta(D^\circ) \le \omega$ by Theorem \ref{tm:e}.
This means that there is an open set $U$ in $Z$ such that $0 < |U \cap D^\circ| \le \omega$.

Pick a point $x \in U \cap D^\circ$ and an open set $V$ such that $x \in V \subs \overline{V} \subs U$.
Then for $E = D \cap V$ we have $|E|  = \omega_1$, moreover $E^\circ \subs \overline{V} \cap D^\circ \subs U \cap D^\circ$
implies that $E^\circ$ is countable, i.e. $\overline{E} \in \mathcal{Z}_3$.

\end{proof}

\begin{corollary}\label{cor:piZ3}
If $X \in \widetilde{\mathcal{E}}$ satisfies $\Delta(X) > \omega$ and $\mathfrak{R}_\omega(X) = \emptyset$
then $\{Z \subs X : Z \in \mathcal{Z}_3\}$ is a $\pi$-network in $X$.
\end{corollary}

\begin{proof}
Theorem \ref{tm:Pt} and $\mathfrak{R}_\omega(X) = \emptyset$ imply that $T = \{x \in X : t(x,X) > \omega\}$ is dense in $X$.
Also, by Lemma \ref{lm:goodpartition} the open subsets of $X$ admitting a good partition form a $\pi$-base $\mathcal{G}$ in $X$.
For every $U \in \mathcal{G}$ there is a point $x \in T \cap U$. We may then pick an open set $V$ such that $x \in V \subs \overline{V} \subs U$.
Then $t(x,X) = t(x,\overline{V}) = \omega_1$ implies the existence of a set $A \subs \overline{V}$ that is
$\omega$-closed but not closed because $x \in \overline{A} \setm A$.
$X \in \widetilde{\mathcal{E}}$ then implies the existence of some $Y \in \mathcal{E}(X)$
such that $A \cap Y$ is not closed in $Y$. By Lemma \ref{lm:eZ1} then $W = \overline{A \cap Y} \cap Y \in \mathcal{Z}_1$.

We also have $W \subs \overline{A} \subs \overline{V} \subs U \in \mathcal{G}$, which implies that
$W$ admits a good partition, hence we may apply Lemma \ref{lm:Z1Z3} to $W$
to find $Z \in \mathcal{Z}_3$ such that $Z \subs W \subs U$, and this completes our proof.

\end{proof}

Before presenting our final -- and crucial -- lemma, we need a definition.

\begin{definition}\label{def:Y(X)}
For any space $X$ we shall denote by $\mathcal{Y}(X)$ the family of all those subspaces
$Y \subs X$ that satisfy conditions (i) - (iii) below:

\begin{enumerate}[(i)]
\item $Y$ is $\omega$-closed in $X$;

\smallskip

\item if $Z \in \mathcal{Z}_3(X)$ is such that $I(Z) \subs Y$ then $Z \cap Y \in \mathcal{Z}_3$;

\smallskip

\item $\mathcal{Z}_3(Y) = \{Z \subs Y : Z \in \mathcal{Z}_3\}$ is a $\pi$-network in $X$.

\end{enumerate}
\end{definition}

Corollary \ref{cor:piZ3} implies that if $X \in \widetilde{\mathcal{E}}$ satisfies both
$\Delta(X) > \omega$ and $\mathfrak{R}_\omega(X) = \emptyset$,
then $X \in \mathcal{Y}(X)$. On the other hand, it is obvious from (iii) that $\mathcal{Y}(X) \ne \emptyset$
implies $\Delta(X) > \omega$.

\begin{lemma}\label{lm:Y(X)}
Assume that $X \in \widetilde{\mathcal{E}}$ and $\mathfrak{R}_\omega(X) = \emptyset$. Then
for every $Y \in \mathcal{Y}(X)$ there is $W \in \mathcal{Y}(X)$ such that $W \subs Y$ and
$Y \setm W$ is dense in $X$.
\end{lemma}

\begin{proof}
We may fix for every $L \in \mathcal{Z}_3(Y)$ a subset $H_L$ which is $\omega_1$-approachable in $L$, as well as
a disjoint family $\{X_{L,\alpha} : \alpha < \omega_1\} \subs [I(L)]^{\omega_1}$
witnessing the $\omega_1$-approachability of $H_L$ in $L$. Next, let $\mathcal{K}$ be a {\em maximal} subfamily of $\mathcal{L}$
such that if $K,\,L$ are distinct members of $\mathcal{K}$ then $|X_{K,\alpha} \cap X_{L,\beta}| \le \omega$ for all $\alpha, \beta < \omega_1$.
By Theorem \ref{tm:pinw} then $S = \bigcup \{H_K : K \in \mathcal{K}\}$ is dense in $X$. This means that $U \cap S \ne \emptyset$ for all
$U \in \tau^+(X)$, however much more is true: $|U \cap S| > \omega_1$ for all $U \in \tau^+(X)$.

Indeed, $|U \cap S| \le \omega_1$ would imply the existence of $\mathcal{M} \subs \mathcal{K}$ with $|\mathcal{M}| \le \omega_1$
such that $U \cap S \subs \bigcup \{H_M : M \in \mathcal{M}\}$. But by Theorem \ref{tm:main} then we had $\omega_1$ many
pairwise disjoint sets $D_\alpha \subs \bigcup \mathcal{M}$ such that
$U \cap S \subs \bigcup \{H_M : M \in \mathcal{M}\} \subs \overline{D_\alpha}$ for all $\alpha < \omega_1$.
But then $\{U \cap D_\alpha : \alpha < \omega_1\}$ are disjoint dense subsets of $U$,
contradicting $\mathfrak{R}_\omega(X) = \emptyset$.

We claim that the $\omega$-closure $W = \bigcup \{\overline{A} : A \in [S]^\omega\}$ of $S$ in $X$ is as required.
Now, $W \subs Y$ and (i) for $W$ are obvious.

To prove (ii), consider any $Z \in \mathcal{Z}_3(X)$ such that $I(Z) \subs W$, then $Y \in \mathcal{Y}(X)$
implies $Z \cap Y \in \mathcal{Z}_3$. Thus for any $E \in [I(Z)]^{\omega_1}$ we clearly have
$L = \overline{E} \cap Z \cap Y \in \mathcal{Z}_3(Y)$. But then there are
$K \in \mathcal{K}$ and $\alpha, \beta < \omega_1$ for which $C = X_{K,\alpha} \cap X_{L,\beta} \subs E$
is uncountable. But then $\emptyset \ne H_K \cap C^\circ \subs S \cap E^\circ$, hence as $E \in [I(Z)]^{\omega_1}$
was arbitrary, we have $Z \cap W \in \mathcal{Z}_3$.

Now, to prove (iii), note first that, as $\Delta(W) \ge \Delta(S) > \omega_1$ and $\mathfrak{R}_\omega(X) = \emptyset$, Theorem \ref{tm:Pt} implies that
$\{y \in W : t(y,W) > \omega_1\}$ is dense in $W$. Also, by Lemma \ref{lm:goodpartition} the open subsets of $X$ admitting a good partition
form a $\pi$-base $\mathcal{G}$ in $X$. Consequently, it will suffice to show that every $U \in \mathcal{G}$ includes some member of $\mathcal{Z}_3(W)$.

To see this, consider any $U \in \mathcal{G}$ and pick a point $p \in U \cap W$ such that $t(p, U \cap W) > \omega$ and an open set $V$
such that  $p \in V \subs \overline{V} \subs U$. As we have seen above several times, then there is a set $A \subs \overline{V} \cap W$
that is $\omega$-closed in $W$ and hence in $X$ but not closed in $W$ and hence in $X$. Then $X \in \widetilde{\mathcal{E}}$ implies
the existence of some $L \in \mathcal{E}(X)$ such that $A \cap L$ is not closed in $L$. By Lemma \ref{lm:eZ1} then we have
$M = \overline{A \cap L} \cap L \in \mathcal{Z}_1$, moreover $A \cap L$ is an $\omega$-closed proper dense subset of $M$.
Then $M \subs U$ implies that $M$ admits a good partition and we also have $\mathfrak{R}_{\omega}(M) = \emptyset$, hence Lemma \ref{lm:Z1Z3}
implies the existence of some $Z \in \mathcal{Z}_3(M)$ such that $I(Z) \subs A \cap L \subs W$. But then by (ii) we have
$Z \cap W \in \mathcal{Z}_3(W)$ and $Z \cap W \subs M \subs U$.

It remains to prove that $Y \setm W$ is dense in $X$. Assume, on the contrary, that there is some $U \in \tau^+(X)$ such that
$U \cap Y \subs W$. We may then fix for each point $y \in U \cap Y \cap S$ a member $K_y \in \mathcal{K}$ such that $y \in H_{K_y}$
and for each point $y \in U \cap Y \setm S$ a countable set $A_y \subs U \cap S$ such that $y \in \overline{A_y}$.

We now define sets $Q_n \in [U \cap Y]^{\omega_1}$ for $n < \omega$ as follows. $Q_0 \in [U \cap Y]^{\omega_1}$ is chosen arbitrarily.
Then, if $Q_n \in [U \cap Y]^{\omega_1}$ has been defined then we put
$$Q_{n+1} = Q_n \cup \bigcup \{U \cap K_y : y \in  Q_n \cap S\} \cup \bigcup \{A_y : y \in Q_n \setm S\}\,.$$
A straight forward induction then yields $Q_n \in [U \cap Y]^{\omega_1}$ for all $n < \omega$.
Finally we put $Q = \cup_{n < \omega} Q_n$.

Then $Q \cap S$ is dense in $Q$ because for every $y \in Q \setm S$ we have $A_y \subs Q \cap S$.
Moreover, for every $y \in Q \cap S$ we have $U \cap K_y \subs Q$.

Now consider the collection $\mathcal{M} = \{K_y : y \in Q \cap S\}$. Then $|\mathcal{M}| \le \omega_1$,
hence Theorem \ref{tm:main} applies and yields us disjoint sets
$D_\alpha \subs \bigcup \mathcal{M}$ for $\alpha < \omega_1$ such that
$\bigcup \{H_M : M \in \bigcup \mathcal{M}\} \subs \overline{D_\alpha}$ for all $\alpha < \omega_1$.

For every $y \in Q \cap S$ we have $y \in H_{K_y} \subs \overline{D_\alpha}$, hence as $U$ is open,
$y \in \overline{U \cap D_\alpha}$ as well. But for each $M \in \mathcal{M}$ we have $U \cap M \subs Q$,
hence $D_\alpha \subs \bigcup \mathcal{M}$ implies $U \cap D_\alpha \subs Q$. But this would imply
that $Q$ is $\omega_1$-resolvable, contradicting that $\mathfrak{R}_\omega(X) = \emptyset$.

\end{proof}

Now we are ready to prove Theorem \ref{tm:c.e.-generated}.

\begin{proof}[Proof of Theorem \ref{tm:c.e.-generated}]

\end{proof}
Since $\widetilde{\mathcal{E}}$ is open hereditary,
by Theorem \ref{tm:el} it suffices to prove that for every $X \in \widetilde{\mathcal{E}}$ with $\Delta(X) > \omega$
we have $\mathfrak{R}_\omega(X) \ne \emptyset$. Assume, on the contrary, that $\mathfrak{R}_\omega(X) = \emptyset$
for some such a space $X$.

Then, as we have noted above, $X \in \mathcal{Y}(X)$. We may then define a sequence $\{Y_n : n < \omega\} \subs \mathcal{Y}(X)$
as follows. Put $Y_0 = X$, and if $Y_n \in \mathcal{Y}(X)$ has been defined then we apply Lemma \ref{lm:Y(X)} to obtain
$Y_{n+1} \in \mathcal{Y}(X)$ such that $Y_{n+1} \subs Y_n$ and $Y_n \setm Y_{n+1}$ is dense in $X$.
But then $\{Y_n\setm Y_{n+1}:n<{\omega}\}$ is a family of  pairwise disjoint dense sets in $X$,
contradicting $\mathfrak{R}_\omega(X) = \emptyset$ and thus completing our proof.

\end{document}